\documentclass{IEEEtran}
\pdfoutput=1
\usepackage{cite}
\usepackage{amsmath,amssymb,amsfonts,amsthm,mathtools}
\usepackage{algorithmic}
\usepackage{graphicx}
\usepackage{textcomp}
\usepackage{xspace}
\def\BibTeX{{\rm B\kern-.05em{\sc i\kern-.025em b}\kern-.08em
    T\kern-.1667em\lower.7ex\hbox{E}\kern-.125emX}}
    
\newcommand{\calA}{\mathcal{A}}
\newcommand{\calB}{\mathcal{B}}
\newcommand{\calC}{\mathcal{C}}

\newcommand{\calH}{\mathcal{H}}

\newcommand{\calL}{\mathcal{L}}

\newcommand{\calQ}{\mathcal{Q}}

\newcommand{\calW}{\mathcal{W}}
\newcommand{\calX}{\mathcal{X}}
\newcommand{\calY}{\mathcal{Y}}
\newcommand{\calZ}{\mathcal{Z}}

\newcommand{\R}{\ensuremath\mathbb{R}}
\newcommand{\C}{\ensuremath\mathbb{C}}

\newcommand{\Ltwo}{L^2}

\newcommand{\GL}[2]{\mathrm{GL}_{#1}}

\newcommand{\spd}[1]{\mathcal{S}^{#1}_{\succ}}
\newcommand{\spsd}[1]{\mathcal{S}^{#1}_{\succeq}}

\newcommand{\ds}{\,\mathrm{d}s}
\newcommand{\dt}{\,\mathrm{d}t}

\newcommand{\ddt}{\ensuremath{\tfrac{\mathrm{d}}{\mathrm{d}t}}}

\DeclareMathOperator{\diag}{diag}
\DeclareMathOperator{\dom}{dom}
\DeclareMathOperator{\Kern}{ker}

\DeclareMathOperator{\image}{im}

\DeclareMathOperator{\rank}{rank}

\DeclareMathOperator{\Skew}{skew}
\DeclareMathOperator{\Sym}{sym}

\newcommand{\state}{x}

\newcommand{\stateDim}{n}
\newcommand{\inpVar}{u}
\newcommand{\inpVarDim}{m}
\newcommand{\outVar}{y}

\newcommand{\hamiltonian}{\calH}
\newcommand{\abstractState}{z}
\newcommand{\dstate}{\phi}
\newcommand{\delayHam}{\Theta}

\newcommand{\pH}{\textsf{pH}\xspace}
\newcommand{\KYP}{\textsf{KYP}\xspace}

\newcommand{\DAEs}{\textsf{DAEs}\xspace}
\newcommand{\DDE}{\textsf{DDE}\xspace}
\newcommand{\DDEs}{\textsf{DDEs}\xspace}

\usepackage{graphicx}
\usepackage{color}
\usepackage[dvipsnames]{xcolor}

\usepackage[colorlinks,linkcolor=teal,citecolor=teal,urlcolor=teal]{hyperref}
\usepackage[nameinlink,noabbrev]{cleveref}

\theoremstyle{plain}
\newtheorem{theorem}{Theorem}[section]
\newtheorem{proposition}[theorem]{Proposition}
\newtheorem{lemma}[theorem]{Lemma}

\newtheorem{remark}[theorem]{Remark}
\newtheorem{definition}[theorem]{Definition}

\newtheorem{example}[theorem]{Example}

\begin{document}
\title{Towards a modeling class for port-Hamiltonian systems with time-delay}
\author{Tobias Breiten, Dorothea Hinsen, and Benjamin Unger
\thanks{T.~Breiten and D.~Hinsen are with the Department of Mathematics, Technische Universität Berlin, Berlin, Germany (e-mails: \{breiten,hinsen\}@math.tu-berlin.de).}
\thanks{B.~Unger is with the Stuttgart Center for Simulation Science (SimTech), University of Stuttgart, Germany (e-mail: benjamin.unger@simtech.uni-stuttgart.de).}
\thanks{T.~Breiten is partially supported by the BMBF (grant no.~05M22KTB) and the DFG within the CRC 154 (239904186). 
D. Hinsen acknowledges funding from the DFG within CRC 910.
B. Unger is partially supported by the BMBF (grant no.~05M22VSA), the DFG within EXC 2075 (390740016), and is thankful for support by the Stuttgart Center for Simulation Science (SimTech).}}

\maketitle

\begin{abstract}
	The framework of port-Hamiltonian (pH) systems is a powerful and broadly applicable modeling paradigm. In this paper, we extend the scope of pH systems to time-delay systems. Our definition of a delay pH system is motivated by investigating the Kalman-Yakubovich-Popov inequality on the corresponding infinite-dimensional operator equation. Moreover, we show that delay pH systems are passive and closed under interconnection. We describe an explicit way to construct a Lyapunov-Krasovskii functional and discuss implications for delayed feedback.
\end{abstract}


\vspace{-2em}
\section{Introduction}
\label{sec:intro}

The \emph{port-Hamiltonian} (\pH) framework \cite{JacZ12,SchJ14} constitutes an innovative energy-based model paradigm that offers a systematic approach to the interaction of (physical) systems with each other and the environment via interconnection structures. The inherent structure of \pH systems encodes, among other advantages, control theoretical concepts such as passivity and stability, and facilitates structure-preserving approximation schemes \cite{Egg19,BreU22}. It applies to a large range of different applications and was recently extended to descriptor systems; see \cite{MehU22-ppt} and the references therein. Nevertheless, a consistent generalization to infinite-dimensional systems is not yet available. In particular, extending the \pH concept to \emph{delay differential-algebraic equations} (\DAEs) poses significant challenges with only limited contributions in the literature that primarily focus on the stability analysis of \pH systems with delayed feedback, cf.~\cite{Sun11,KaoP12,YanW13,AouLES14,SchFOR16,MatMN20}. 
A notable exception is provided in \cite{Kur22}, which, however, deals mainly with the solvability analysis. For an overview of \DDEs and their applications, we refer to \cite{Ern09,Ung20b}.

In this paper, we study how the \pH framework can be extended to linear time-invariant \emph{delay differential equations}~(\DDEs) of the form
\begin{equation}
	\label{eqn:delaySys}
	\left\{\ \begin{aligned}
	\dot{\state}(t) &= A_0 \state(t) + A_1\state(t-\tau) + B\inpVar(t), & t > 0,\\
	\outVar(t) &= C\state(t), & t>0,\\
	\state(t) &= \phi(t), & t\in[-\tau,0],
	\end{aligned}\right.
\end{equation}
where $A_0,A_1\in\R^{\stateDim\times\stateDim}$, $B,C^\top\in\R^{\stateDim\times\inpVarDim}$. In the non-delayed case, which is obtained if either $\tau=0$ or $A_1=0$ in~\eqref{eqn:delaySys}, and some further technical details, the system~\eqref{eqn:delaySys} has a \pH representation if and only if~\eqref{eqn:delaySys} is passive \cite{BeaMX22,CheGH22}. In this case, a solution of the associated \emph{Kalman-Yakubovich-Popov} (\KYP) inequality can then be used to construct the explicit \pH representation; see \cref{subsec:pH} for further details. We follow this idea by recasting the \DDE~\eqref{eqn:delaySys} as an infinite-dimensional linear system without delay \cite{CurZ20}, for which we study special solutions of the corresponding operator \KYP inequality. The \pH structure on the operator level is then recast as a \DDE, which forms the basis for our definition of a time-delay \pH system. Our main results are the following:
\begin{enumerate}
	\item We present a novel definition of a time-delay \pH system in \Cref{def:pHDelay} and prove that time-delay \pH systems are passive and closed under interconnection, see \Cref{lem:dissipationInequality,lem:interconnection}.
	\item The condition for a \DDE to be \pH generalizes a well-known matrix inequality for delay-independent passivity (cf.~\Cref{prop:delayPHinequality}) and details the construction of the necessary Lyapunov-Krasovskii functional in \Cref{lem:pHdelayCondition:RZrelation} and \Cref{thm:LyapKravoskiiConstruction}. Implications for time-delayed feedback are discussed in \cref{subsec:delayedFeedback}.
\end{enumerate}

We review the necessary definitions and preliminary results in \cref{sec:preliminaries} and present our main results in \cref{sec:timeDelayPH}.

\subsection*{Notation} 
The sets of nonsingular matrices, symmetric positive-definite, and symmetric positive semi-definite matrices of dimension $\stateDim$ are denoted by $\GL{\stateDim}{\R}$, $\spd{\stateDim}$, and $\spsd{\stateDim}$, respectively. Moreover, for a matrix $F\in\R^{\stateDim\times\stateDim}$, we use the notation $\Sym(F)\vcentcolon= \tfrac{1}{2}(F+F^\top)$ and $\Skew(F)\vcentcolon= \tfrac{1}{2}(F-F^\top)$.
The set of all linear and bounded operators from a Banach space $\calX$ to a Banach space $\calY$ is denoted by $\calL(\calX,\calY)$. If $\calX = \calY$, we simply write $\calL(\calX)$.

\section{Preliminaries}
\label{sec:preliminaries}

In this section, we recall important results on \pH systems and time-delay systems that we will later leverage to motivate the definition of a time-delay \pH system.

\subsection{Passive and port-Hamiltonian systems}
\label{subsec:pH}

We consider a linear time-invariant system of the form
\begin{equation}
	\label{eqn:LTIsys}
	\Sigma\quad \left\{\quad\begin{aligned}
		\dot{\state}(t) &= A \state(t) + B\inpVar(t),\\
		\outVar(t) &= C\state(t),
	\end{aligned}\right.
\end{equation}
with $\inpVar\colon \R \to \R^\inpVarDim$, $\state\colon \R \to \R^\stateDim$, $\outVar \colon \R \to \R^\inpVarDim$ are the \emph{input}, \emph{state}, and \emph{output} of the system and matrices $A\in\R^{\stateDim\times\stateDim}$, $B\in\R^{\stateDim\times\inpVarDim}$, and $C\in\R^{\inpVarDim\times\stateDim}$. Throughout this subsection,  we assume that~\eqref{eqn:LTIsys} is \emph{minimal}, i.e., for all $s\in\C$ the conditions
\begin{align*}
	\rank\begin{bmatrix}
		sI_n-A, & B
	\end{bmatrix} = n =
	\rank\begin{bmatrix}
		sI_n-A^\top, & C^\top
	\end{bmatrix}
\end{align*}
are satisfied. 
For convenience, we introduce for $H\in\R^{\stateDim\times\stateDim}$ the notation
\begin{equation}
	\label{eqn:pHMatrixDef}
	\Sigma_H \vcentcolon= \begin{bmatrix}
		HA & HB\\
		-C & 0
	\end{bmatrix}\in\R^{(\stateDim+\inpVarDim) \times (\stateDim+\inpVarDim)}.
\end{equation}
Following \cite{Wil72b,MehU22-ppt}, we give the following definitions.

\begin{definition}
	\label{def:PassPH}
	We consider system~\eqref{eqn:LTIsys}.
	\begin{enumerate}
		\item System~\eqref{eqn:LTIsys} is called \emph{passive}, if there exists a state-dependent \emph{storage function} $\hamiltonian\colon \R^\stateDim\to \R_+$ satisfying for any $t_1\geq t_0$ the dissipation inequality
			\begin{equation}
				\label{eqn:dissipationInequality}
				\hamiltonian(\state(t_1)) - \hamiltonian(\state(t_0)) \leq \int_{t_0}^{t_1} \outVar(t)^\top \inpVar(t)\dt
			\end{equation}
			for trajectories $(\inpVar,\state,\outVar)$ satisfying~\eqref{eqn:LTIsys}. 
		\item System~\eqref{eqn:LTIsys} is called \emph{port-Hamiltonian} (\pH), if there exists $H = H^\top \in \spd{\stateDim}$ such that the \emph{dissipativity condition}
			\begin{equation}
				\label{eqn:dissipativityCondition}
				\Sym\left(\Sigma_H\right) \leq 0.
			\end{equation}	
			is satisfied. In this case, we define the matrices
			\begin{equation*}
				\label{eqn:pH:matrixDefinitions}
				\begin{bmatrix}
					R & 0\\
					0 & 0
				\end{bmatrix} \vcentcolon= -\Sym(\Sigma_H),\  
				\begin{bmatrix}
					J & G\\
					-G^\top & 0
				\end{bmatrix} \vcentcolon= \Skew(\Sigma_H),
			\end{equation*}								
			and call
			\begin{equation}
				\label{eqn:pH}
				\left\{~\begin{aligned}
		H\dot{\state}(t) &= (J-R)\state(t) + G\inpVar(t),\\
		\outVar(t) &= G^\top \state(t)
	\end{aligned}\right.
			\end{equation}
			a (generalized state-space) \emph{\pH representation} of~\eqref{eqn:LTIsys} with associated Hamiltonian
			\begin{equation}
				\label{eqn:pH:Hamiltonian}
				\hamiltonian(\state) \vcentcolon= \tfrac{1}{2} \state^\top H\state.
			\end{equation}
	\end{enumerate}
\end{definition}

Straightforward calculations show that a \pH system is passive with the Hamiltonian acting as storage function.
Conversely, assume that~\eqref{eqn:LTIsys} is passive, minimal, and stable. Define $\mathcal{W}\colon \R^{\stateDim\times \stateDim} \to \R^{(\stateDim+\inpVarDim)\times (\stateDim+\inpVarDim)}$ via
\begin{equation}
	\label{eqn:KYPmatrix}
	\calW(H) \vcentcolon= \begin{bmatrix}
		-A^\top H - HA & C^\top - HB\\
		C - B^\top H & 0
	\end{bmatrix}.
\end{equation}
It is well-known, see for instance \cite{Wil72b}, that the \KYP inequality
\begin{equation}
	\label{eqn:KYP}
	\calW(H) \geq 0
\end{equation}
has a solution $H\in\spd{\stateDim}$. We immediately notice that~\eqref{eqn:KYP} resembles the  dissipativity condition~\eqref{eqn:dissipativityCondition}. Thus, we have proven the following equivalence, which is already well-known in the literature, see for instance \cite{BeaMX22}.

\begin{theorem}
	\label{thm:equivalence}
	Assume that~\eqref{eqn:LTIsys} is minimal and stable. Then the following are equivalent:
	\begin{enumerate}
		\item The system~\eqref{eqn:LTIsys} is passive.
		\item The system~\eqref{eqn:LTIsys} is port-Hamiltonian.
	\end{enumerate}
\end{theorem}

\begin{remark}
	\label{rem:generalPHformat}
	In the literature, see for instance \cite{DuiMSB09,SchJ14,MehU22-ppt}, \pH systems are often modelled in a slightly different form. Such a different formulation can be achieved via a factorization of the Hessian of the Hamiltonian of the form $H = E^\top Q\in\spsd{\stateDim}$ with matrices $E,Q\in\R^{\stateDim\times\stateDim}$. Then, a system of the form  
	\begin{equation}
		\label{eqn:pHdescriptor}
			\left\{~\begin{aligned}
			E\dot{\state}(t) &= (J-R)Q\state(t) + G\inpVar(t),\\
			\outVar(t) &= G^\top Q\state(t)
	\end{aligned}\right.
	\end{equation}
	is called a \pH system, if $J=-J^\top$ and $Q^\top RQ \in\spsd{\stateDim}$.
	Depending on the modelling, the factorization is often chosen as $(E,Q) = (I_\stateDim,H)$ or $(E,Q) = (H,I_{\stateDim})$ (as we do in \Cref{def:PassPH}).
\end{remark}

\subsection{Time-delay systems as abstract infinite-dimensional systems}
\label{subsec:timeDelay:operator}
To recast the delay equation~\eqref{eqn:delaySys} as abstract operator differential equation, we follow \cite[Sec.~3.3]{CurZ20} and define the Hilbert space $\calZ_{\stateDim;\tau}\vcentcolon= \R^\stateDim \times \Ltwo([-\tau,0]; \R^\stateDim)$ with standard inner product
\begin{equation*}
	\left\langle \begin{bmatrix}
		\state_1\\ \dstate_1
	\end{bmatrix}, \begin{bmatrix}
		\state_2\\ \dstate_2
\end{bmatrix}\right\rangle_{\!\!\calZ_{\stateDim;\tau}} \!\!\vcentcolon= \langle \state_1,\state_2 \rangle_{\R^\stateDim} + \langle \dstate_1, \dstate_2 \rangle_{\Ltwo}.
\end{equation*}
We define the linear operator $\calA\colon\dom(\calA)\subseteq\calZ_{\stateDim;\tau} \to\calZ_{\stateDim;\tau}$, $\calB\colon\R^{\inpVarDim}\to\calZ_{\stateDim;\tau}$, and $\calC\colon\calZ_{\stateDim;\tau}\to\R^{m}$ via
\begin{align}
	\calA\begin{bmatrix}
		\state\\\phi
	\end{bmatrix} \vcentcolon= \begin{bmatrix}
		A_0 \state + A_1 \dstate(-\tau)\\
		\tfrac{\mathrm{d}}{\mathrm{d}s}\dstate
	\end{bmatrix},\ 
	\calB\inpVar &\vcentcolon= \begin{bmatrix}
		B\inpVar\\
		0
	\end{bmatrix},\ 
	\calC\begin{bmatrix}
		\state\\\dstate
	\end{bmatrix} \vcentcolon= C\state
\end{align}
with domain
\begin{align*}
	\resizebox{\linewidth}{!}{$
	\dom(\calA) \vcentcolon= \left\{ \begin{bmatrix}
		\state\\\dstate
	\end{bmatrix} \in \calZ_{\stateDim;\tau} \,\,\left|\,\, \begin{aligned} & \dstate \text{ is absolutely continuous,}\\
	&\tfrac{\mathrm{d}}{\mathrm{d}s} \dstate \in L_2([-\tau;0]; \R^\stateDim), \text{ and } \dstate(0)=\state\end{aligned} \right.\right\}.$}
\end{align*}
With these preparations, we can study the operator equation
\begin{equation}
	\label{eqn:delaySys:operatorForm}
	\left\{\quad\begin{aligned}
	\dot{\abstractState}(t) &= \calA \abstractState(t) + \calB \inpVar(t), & t>0,\\
	\outVar(t) &= \calC \abstractState(t), & t>0,\\
	\abstractState(0) &= \abstractState_0,
	\end{aligned}\right.
\end{equation}
and it is well-known from the literature, that we can study~\eqref{eqn:delaySys:operatorForm} instead of~\eqref{eqn:delaySys}.

\subsection{Passive time-delay systems}
\label{subsec:timeDelay:passive}

Delay-independent passivity is typically studied via a Lyapunov-Krasovskii type functional of the form
\begin{align}
	\label{eqn:LyapunovKrasovskiiFunctional}
	\hamiltonian(\state\big|_{[t-\tau,t]}) = \state(t)^\top Q\state(t) + \int_{t-\tau}^t \state(s)^\top \delayHam\state(s)\ds,
\end{align}
denotes the classical function segment used in the delay literature; cf.~\cref{subsec:timeDelay:operator}. The following result, taken from \cite[Lem.~1]{NicL01}, provides a sufficient condition for passivity of~\eqref{eqn:delaySys}. Note that we present the result with a non-strict inequality, which does not influence the line of reasoning within the proof of \cite[Lem.~1]{NicL01}.

\begin{lemma}
	\label{lem:passiveTimeDelay}
	Consider the delay equation~\eqref{eqn:delaySys}. If there exist positive definite matrices $Q,\delayHam\in\spd{\stateDim}$ such that
	\begin{subequations}
		\begin{align}
		&A_0^\top Q + Q A_0 + Q A_1 \delayHam^{-1} A_1^\top Q + \delayHam \leq 0, \label{eqn:NicL01inequality}\\
		&C = B^\top Q,
	\end{align} 
	\end{subequations}
	then the Lyapunov-Krasovskii function~\eqref{eqn:LyapunovKrasovskiiFunctional} satisfies the dissipation inequality
	\begin{align}
		\label{eqn:NicL01passivityinequ}
		\tfrac{1}{2}\hamiltonian(\state\big|_{[t_1-\tau,t_1]}) - \tfrac{1}{2}\hamiltonian(\state\big|_{[t_0-\tau,t_0]}) \leq \int_{t_0}^{t_1} \outVar(t)^\top \inpVar(t)\dt
	\end{align}
	and hence, the delay equation~\eqref{eqn:delaySys} is passive.
\end{lemma}

\section{Time-delay port-Hamiltonian systems}
\label{sec:timeDelayPH}

Leveraging the previous section's discussion, we take the following route to obtain a meaningful definition of a time-delay \pH system. First, we use the operator formulation~\eqref{eqn:delaySys:operatorForm} of the time-delay system and assume that the associated operator \KYP inequality has a (special) solution. Testing the operator \KYP inequality with suitable test functions, we obtain a finite-dimensional passivity condition that we can use to formulate a time-delay \pH system. The details are presented in \cref{subsec:timeDelayPH:def}.

\subsection{Motivation and definition}
\label{subsec:timeDelayPH:def}

We briefly recall the concepts of passivity of infinite-dimensional linear systems for which we closely follow the exposition from \cite[Sec.~7.5]{CurZ20}.
Given a self-adjoint, nonnegative operator $\calQ\in \calL(\calZ_{\stateDim;\tau})$, we call the system \eqref{eqn:delaySys:operatorForm} \emph{impedance passive} with respect to the \emph{supply rate} $s(\inpVar,\outVar) \vcentcolon= \langle \outVar,\inpVar\rangle + \langle \inpVar,\outVar\rangle$ 
and the \emph{storage function} $\hamiltonian(\abstractState)=\langle \abstractState,\calQ \abstractState\rangle$ if for all $t>0, \abstractState_0 \in \calZ_{\stateDim;\tau}$ and $\inpVar\in L^2([0,\tau];\R^\inpVarDim)$ it holds that
\begin{align*}
 \hamiltonian(\abstractState(t))\le \hamiltonian(\abstractState_0)+\int_0^t s(\inpVar(\theta),\outVar(\theta))\,\mathrm{d}\theta.
\end{align*}

From \cite[Thm.~7.5.3]{CurZ20} it follows that \eqref{eqn:delaySys:operatorForm} is impedance passive if and only if for $\inpVar \in \R^\inpVarDim$ and $\abstractState \in \dom(\calA)$ the following linear operator inequality is satisfied:
\begin{multline}\label{eqn:passive_lin_op_in}
 \langle \calA\abstractState,\calQ\abstractState \rangle +\langle \calQ\abstractState,\calA\abstractState\rangle\\
 + \left\langle \begin{bmatrix} \abstractState \\ \inpVar \end{bmatrix} , \begin{bmatrix} 0& \calQ\calB-\calC^* \\ \calB^*\calQ - \calC & 0 \end{bmatrix} \begin{bmatrix} \abstractState \\ \inpVar\end{bmatrix} \right\rangle \le 0.
\end{multline}
Note that if $\calQ$ maps from $\dom(\calA)$ into $\dom(\calA^*)$ instead of \eqref{eqn:delaySys:operatorForm} we may introduce
\begin{equation}
	\calQ_{\mathrm{ext}} \vcentcolon= \begin{bmatrix}
		\calA^* \calQ + \calQ\calA & \calQ\calB - \calC^*\\
		\calB^*\calQ - \calC & 0
	\end{bmatrix},
\end{equation}
and consider a KYP-type inequality (cf.~\cite[Lem~7.5.4]{CurZ20})
\begin{equation}
	\label{eqn:KYP:operator}
	\left\langle\begin{bmatrix}
		\abstractState\\
		\inpVar
	\end{bmatrix}, \calQ_{\mathrm{ext}} \begin{bmatrix}
		\abstractState\\\inpVar
	\end{bmatrix}\right\rangle \leq 0
\end{equation}
for any $\abstractState\in \dom(\calA)$ and $\inpVar\in\R^{\inpVarDim}$. While this is possible if~\eqref{eqn:passive_lin_op_in} holds with equality, in the following we focus on the inequality~\eqref{eqn:passive_lin_op_in} and do not assume $\calQ$ to map from $\dom(\calA)$ into $\dom(\calA^*)$.

Since $\calQ\in \calL(\calZ_{\stateDim;\tau})$, let us consider the following partitioning 
\begin{equation}
	\label{eqn:formofQcomp}
	\calQ = \begin{bmatrix}
		Q_{11} & \calQ_{12}\\
		\calQ_{12}^* & \calQ_{22}
	\end{bmatrix}, \qquad \begin{aligned}
		Q_{11}&\in\spsd{\stateDim},\\
		\calQ_{12}&\in \calL(L^2([-\tau,0];\R^\stateDim),\R^\stateDim),\\
		\calQ_{22}&\in \calL(L^2([-\tau,0];\R^\stateDim)).
	\end{aligned}
\end{equation}
For $\inpVar=0$ and $\abstractState=(\state,\dstate)\in\dom(\calA)$ with $\dstate(0)=\state$ we obtain from \eqref{eqn:passive_lin_op_in} 
\begin{align*}
		0 &\geq\left\langle \calA \begin{bmatrix}
			\state\\ \dstate
		\end{bmatrix}, \calQ \begin{bmatrix}
		\state\\ \dstate
	\end{bmatrix}\right\rangle + \left\langle \calQ \begin{bmatrix}
	\state\\ \dstate
\end{bmatrix}, \calA \begin{bmatrix}
\state\\ \dstate
\end{bmatrix}\right\rangle\\
 &=	\left\langle \begin{bmatrix}
		A_0 \state + A_1 \dstate(-\tau)\\
		\tfrac{\mathrm{d}\dstate}{\mathrm{d}s}
	\end{bmatrix}, \begin{bmatrix}
	Q_{11}\state +\calQ_{12}\dstate\\
	\calQ_{12}^*\state + \calQ_{22}\dstate
\end{bmatrix} \right\rangle\\
&\phantom{=}\quad
+ \left\langle \begin{bmatrix}	
	Q_{11}\state +\calQ_{12}\dstate\\
	\calQ_{12}^*\state + \calQ_{22}\dstate
\end{bmatrix}, \begin{bmatrix}
A_0 \state + A_1 \dstate(-\tau)\\
\tfrac{\mathrm{d}\dstate}{\mathrm{d}s}
\end{bmatrix} \right\rangle.
\end{align*}
To mimic the Lyapunov-Krasovskii functional~\eqref{eqn:LyapunovKrasovskiiFunctional} let us assume $\calQ_{12}=0$ to obtain
\begin{align*}
	0 &\geq (A_0 \state + A_1\dstate(-\tau))^\top Q_{11}\state + \state^\top Q_{11}(A_0 \state + A_1\dstate(-\tau))\\
	&\quad + \left\langle \tfrac{\mathrm{d}\dstate}{\mathrm{d}s} , \calQ_{22}\dstate \right\rangle + \left\langle  \calQ_{22}\dstate, \tfrac{\mathrm{d}\dstate}{\mathrm{d}s} \right\rangle\\
	&= (A_0 \state + A_1\dstate(-\tau))^\top Q_{11}\state + \state^\top Q_{11}(A_0 \state + A_1\dstate(-\tau)) \\
	&\quad + \left\langle \tfrac{\mathrm{d}\dstate}{\mathrm{d}s} , \calQ_{22}\dstate \right\rangle - \left\langle \tfrac{\mathrm{d}}{\mathrm{d}s}\calQ_{22}\dstate,\dstate \right\rangle + \left\langle \calQ_{22}\dstate, \dstate \right\rangle \big|_{-\tau}^{0}.
\end{align*}
Let us further assume that $\calQ_{22}$ is a multiplication operator induced by $Q_{22}\in\spsd{\stateDim}$ such that $g=\calQ_{22}\psi$ with $\psi\in L^2([-\tau,0];\R^{\stateDim})$ is defined by $g(s)=Q_{22}\psi(s)$ for a.\,e.~$s\in [-\tau,0]$. Consequently, we find 
\begin{align*}
\left\langle \calQ_{22}\dstate, \dstate \right\rangle \big|_{-\tau}^{0}	= \state^\top Q_{22} \state - \dstate(-\tau)^\top Q_{22} \dstate(-\tau)
\end{align*}
and
\begin{multline*}
	0 \geq \state^\top(A_0^\top Q_{11}+ Q_{11}A_0 + Q_{22})\state\\
	 + \dstate(-\tau)^\top A_1^\top Q_{11}\state + \state^\top Q_{11}A_1 \dstate(-\tau)\\
	  - \dstate(-\tau)^\top Q_{22}\dstate(-\tau).
\end{multline*}
For arbitrary $\state,\xi\in\R^{\stateDim}$ let us define $\dstate\colon [-\tau,0]\to \R^{\stateDim}$ by $\dstate(s)=\tfrac{s}{\tau}(\state-\xi)+\state$. Then $(\state,\dstate)\in \dom(\calA)$ and the last equation implies 
\begin{align}\label{eqn:kyp_delay}
 \begin{bmatrix} \state \\ \xi \end{bmatrix}^\top \begin{bmatrix} -A_0^\top Q_{11} -Q_{11}A_0 - Q_{22} & -Q_{11} A_1 \\ -A_1^\top Q_{11} & Q_{22} \end{bmatrix} \begin{bmatrix} \state \\ \xi\end{bmatrix} \ge 0.
\end{align}
Note further that from \eqref{eqn:passive_lin_op_in} we also obtain $\calC=\calB^*\calQ$ which implies $C=B^\top Q_{11}$. Moreover, observe that the above simplifications correspond to the storage function
\begin{align*}
	\hamiltonian(\state,\dstate)=\state^\top Q_{11} \state + \int_{-\tau}^0 \dstate(s)^\top Q_{22} \dstate(s)\,\mathrm{d}s.
\end{align*}

Similar to the non-delay case, instead of an explicit representation of $Q_{11}$ and $Q_{22}$ in \eqref{eqn:kyp_delay}, we may consider a generalized state-space representation leading to the following notion of a \pH delay system.
\begin{definition}
	\label{def:pHDelay}
	A time-delay system of the form
	\begin{subequations}
		\label{eqn:pHDelay}
		\begin{equation}
			\left\{\quad\begin{aligned}
				H\dot{\state}(t) &= (J-R)\state(t) - Z\state(t-\tau) + G\inpVar(t),\\
				\outVar(t) &= G^\top \state(t)
			\end{aligned}\right.
		\end{equation}
	with Hamiltonian
	\begin{equation}
		\label{eqn:pHDelay:Hamiltonian}
		\hamiltonian(\state\big|_{[t-\tau,t]}) = \tfrac{1}{2} \state(t)^\top H \state(t) + \int_{t-\tau}^t \state(s)^\top \delayHam \state(s)\ds
	\end{equation}
	\end{subequations}
	is called a \emph{port-Hamiltonian (\pH) delay system}, if $H\in\spd{\stateDim}$, $\delayHam\in\spsd{\stateDim}$, $J = -J^\top$ and
	\begin{equation}
		\label{eqn:pHcondition}
		\begin{bmatrix}
			R-\delayHam &  \tfrac{1}{2}Z\\
			\tfrac{1}{2}Z^\top & \delayHam
		\end{bmatrix}\in \spsd{2n}.
	\end{equation}
\end{definition}

To demonstrate \Cref{def:PassPH}, we consider the following example taken from \cite{NicL01}. 
\begin{example} 
	For parameters $\alpha_0,\alpha_1\in\R$ and time-delay $\tau>0$ we consider the scalar delay differential equation
	\begin{equation}
		\label{eqn:expl}
		\begin{aligned}
			\dot{\state}(t) &=-\alpha_0 \state(t) - \alpha_1 \state(t- \tau) + \inpVar(t)\\
			\outVar(t) &= \state(t)
		\end{aligned}
	\end{equation}
	with appropriate initial condition defined on $[-\tau,0]$. 
	We set $H \vcentcolon= 1$, $J\vcentcolon=0$, $R \vcentcolon= \alpha_0$, $Z\vcentcolon= \alpha_1$, $G\vcentcolon=1$ and $\delayHam \vcentcolon= \theta$ with $\theta \geq 0$. The matrix in~\eqref{eqn:pHcondition} has the eigenvalues 
	\begin{align*}
		\lambda_{1}, \lambda_2 &= \tfrac{\alpha_0}{2} \pm \sqrt{\tfrac{\alpha_0^2}{4}+ \theta^2 - \alpha_0 \theta + \tfrac{\alpha_1^2}{4}},
	\end{align*}
	which are non-negative if and only if $\theta \alpha_0 \geq \theta^2 + \tfrac{1}{4}\alpha_1^2$.
	This immediately implies $\alpha_0 \geq |\alpha_1|\geq 0$. In particular, we conclude that for any 
	\begin{align*}
		\theta \in \left[\tfrac{\alpha_0}{2}-\sqrt{\tfrac{\alpha_0^2}{4}-\tfrac{\alpha_1^2}{4}}, \tfrac{\alpha_0}{2}+\sqrt{\tfrac{\alpha_0^2}{4}-\tfrac{\alpha_1^2}{4}}\right]
	\end{align*}
	condition~\eqref{eqn:pHcondition} is satisfied and thus~\eqref{eqn:expl} is a \pH delay system.
\end{example}

\subsection{Properties}

\begin{lemma}[Passivity]
	\label{lem:dissipationInequality}
	Consider the \pH delay system~\eqref{eqn:pHDelay}. Then
	\begin{equation}
		\label{eqn:dissipationInequalityDelayPH}
		\ddt \hamiltonian\big(\state\big|_{[t-\tau,t]}\big) \leq \outVar(t)^\top \inpVar(t),
	\end{equation}
	along any solution of~\eqref{eqn:pHDelay}. In particular, the \pH delay system~\eqref{eqn:pHDelay} is passive.
\end{lemma}

\begin{proof}
	We obtain
	\begin{align*}
		&\ddt \hamiltonian\big(\state\big|_{[t-\tau,t]}\big)\\
		&= \state(t)^\top H\dot{\state}(t) + \state(t)^\top \delayHam \state(t) - \state(t-\tau)^\top \delayHam\state(t-\tau)\\
		&= -\begin{bmatrix}
			\state(t)\\
			\state(t-\tau)
		\end{bmatrix}^\top\begin{bmatrix}
			R - \delayHam &  \tfrac{1}{2}Z\\
			\tfrac{1}{2}Z^\top & \delayHam
		\end{bmatrix}\begin{bmatrix}
			\state(t)\\
			\state(t-\tau)
		\end{bmatrix} + \state(t)^\top G\inpVar(t)\\
		&\leq \outVar(t)^\top \inpVar(t).
	\end{align*}
	Integration of~\eqref{eqn:dissipationInequalityDelayPH} yields the dissipation inequality and thus passivity of the delay \pH system.
\end{proof}

We now discuss in which sense \Cref{def:pHDelay} is a generalization of the \pH system definition in \Cref{def:PassPH}. We can consider two special cases of the \pH delay system~\eqref{eqn:pHDelay} to retain classical \pH systems. First, let $Z=0$. To recover the classical Hamiltonian, we set $\delayHam=0$ and then observe that~\eqref{eqn:pHcondition} reduces to $R\in \spsd{n}$, i.e., to the classical condition for \pH systems. On the other hand, if $\tau = 0$, then we can again set $\delayHam = 0$ and notice that~\eqref{eqn:pHcondition} is only satisfied for $Z = 0$. On the other hand, if we require
\begin{displaymath}
	\begin{bmatrix}
		\state\\\state
	\end{bmatrix}^\top \begin{bmatrix}
		R & \tfrac{1}{2}Z\\
		\tfrac{1}{2}Z^\top & 0
	\end{bmatrix}\begin{bmatrix}
		\state\\\state
	\end{bmatrix}\geq 0
\end{displaymath}
for any $\state\in\R^{\stateDim}$, which corresponds to the specific choice used in the proof of \Cref{lem:dissipationInequality}, then we recover the condition $R + \Sym(Z) \in \spsd{\stateDim}$, corresponding to the system
\begin{align*}
	H\dot{\state}(t) &= (J-\Skew(Z) - (R+\Sym(Z))\state(t) + G\inpVar(t)\\
	\outVar(t) &= G^\top \state(t).
\end{align*}

In terms of modelling, an important feature of \pH systems is that they are closed under power-conserving or dissipative interconnection. Consider two \pH delay systems of the form 
\begin{align*}
	H_i\dot{\state}_i(t) &= (J_i-R_i)\state_i(t) - Z_i\state_i(t-\tau)  + G_i\inpVar_i(t),\\
	\outVar_i(t) &= G^\top_i \state_i(t)
\end{align*}
with $H_i\in\spd{\stateDim_i}$, $R_i,\delayHam_i\in\spsd{\stateDim_i}$, $-J_i^\top = J_i\in\R^{\stateDim_i\times\stateDim_i}$ and Hamiltonians
\begin{equation*}
	\hamiltonian_i(\state_i\big|_{[t-\tau,t]}) = \tfrac{1}{2} \state_i(t)^\top H_i \state_i(t) + \int_{t-\tau}^t \state_i(s)^\top \delayHam_i \state_i(s)\ds
\end{equation*}
for $i= 1,2$. Define the aggregated input and output vectors as $\tilde{\inpVar} \vcentcolon=\begin{bmatrix}
	\inpVar_1^\top & \inpVar_2^\top
\end{bmatrix}^\top$ and $\tilde{\outVar} \vcentcolon= \begin{bmatrix}
\outVar_1^\top & \outVar_2^\top
\end{bmatrix}^\top$ and consider an output-feedback of the form
\begin{displaymath}
	\tilde{\inpVar} = F\tilde{\outVar} + w
\end{displaymath}
with $F \in \R^{(\inpVarDim_1 + \inpVarDim_2) \times (\inpVarDim_1 + \inpVarDim_2)}$. By defining $\tilde{\state} \vcentcolon= \begin{bmatrix}
\state_1^\top & \state_2^\top
\end{bmatrix}^\top$ and
\begin{align*}
	\tilde{H} &\vcentcolon= \diag(H_1, H_2), & 
	\tilde{J} &\vcentcolon= \diag(J_1, J_2), \\
	\tilde{R} &\vcentcolon= \diag(R_1, R_2), &
	\tilde{Z} &\vcentcolon= \diag(Z_1, Z_2), \\
	\tilde{G} &\vcentcolon= \diag(G_1, G_2), &
	\tilde{\delayHam} &\vcentcolon= \diag(\delayHam_1, \delayHam_2),
\end{align*}
we obtain the interconnected (closed-loop) system
\begin{equation}
	\label{eqn:intercpHf}
		\begin{aligned}
			\tilde{H}
				\dot{\tilde{\state}}(t)
			&= (\tilde{J} -
			\tilde{R} + \tilde{G} F \tilde{G}^\top)
			\tilde{\state}(t)
			+ \tilde{Z} \tilde{\state}(t-\tau) + \tilde{G} w(t)\\
		\tilde{\outVar}(t)
			 &= \tilde{G}^\top	\tilde{\state} (t)
		\end{aligned} 
\end{equation}
with combined Hamiltonian $\tilde{\hamiltonian} \vcentcolon= \hamiltonian_1 + \hamiltonian_2$ given by
\begin{align*}
	\tilde{\hamiltonian}\left(\tilde{\state}_i\big|_{[t-\tau,t]}\right) = \tfrac{1}{2} \tilde{\state}(t)^\top \tilde{H}\tilde{\state}(t) + \int_{t-\tau}^t \tilde{\state}(s)^\top \tilde{\delayHam}\tilde{\state}(s)\ds.
\end{align*}
We have thus proven the following result, which details that delay \pH systems are closed under interconnection.

\begin{lemma}[Interconnection]
	\label{lem:interconnection}
	Let us consider the interconnected system~\eqref{eqn:intercpHf} with Hamiltonian $\tilde{\hamiltonian} \vcentcolon= \hamiltonian_1 + \hamiltonian_2$. Then, the interconnected system~\eqref{eqn:intercpHf} is a \pH delay system if
	\begin{equation} 
		\label{eqn:pHcondition:interconnection}
		\begin{bmatrix}
			\tilde{R} - \tilde{G} \Sym (F) \tilde{G}^\top -\tilde{\delayHam} & \tfrac{1}{2} \tilde{Z}\\
			\tfrac{1}{2} \tilde{Z}^\top & \tilde{\delayHam}
		\end{bmatrix}\in\spsd{2(\stateDim_1+\stateDim_2)}.
	\end{equation}
	In particular, the interconnected system is a \pH delay system for every power-conserving ($\Sym(F) = 0$) and every dissipative ($-\Sym(F)\in\spsd{\stateDim_1+\stateDim_2}$) interconnection.
\end{lemma}

\subsection{Construction of the Lyapunov-Krasovkii functional}
\label{subsec:timeDelayPH:implications}

While the energy (and thus the Hamiltonian) for non-delayed systems is typically obtained during the modeling process, this may not be the case for the Hamiltonian, i.e., the Lyapunov-Krasovskii functional~\eqref{eqn:LyapunovKrasovskiiFunctional}, of the delay system~\eqref{eqn:delaySys}. In particular, if the delay system results from delayed feedback of a non-delayed \pH system (see the forthcoming \cref{subsec:delayedFeedback}), then the matrix $\delayHam\in\spsd{\stateDim}$ in~\eqref{eqn:pHDelay:Hamiltonian} is not available. We thus discuss in the following conditions when such a matrix $\delayHam$ satisfying the delay \pH condition~\eqref{eqn:pHcondition} exists. We first start with a geometric analysis.

\begin{proposition}
	\label{lem:pHdelayCondition:RZrelation}
	A necessary condition for a system of the form~\eqref{eqn:pHDelay} with $\tau>0$ to be a \pH delay system is
	\begin{subequations}
		\label{eqn:pHcondition:necessaryConditions}
		\begin{align}
		\label{eqn:pHcondition:necessaryConditions:a}&\ker(R) \subseteq \ker(\delayHam) \subseteq \ker(Z),\\ 
		\label{eqn:pHcondition:necessaryConditions:b}&\ker(R) \cap \image(Z) = \{0\}, \text{ and}\\
		\label{eqn:pHcondition:necessaryConditions:c}&\ker(R) \cap \image(\delayHam) = \{0\}	.	
	\end{align}
	\end{subequations}
\end{proposition}

\begin{proof}
	Let $r \vcentcolon= \rank(R)$ and $s\vcentcolon=\rank(\delayHam)$, define $\tilde{Z} \vcentcolon= \tfrac{1}{2}Z$ and assume that the matrices $R,Z$, and $\delayHam$ satisfy~\eqref{eqn:pHcondition}. Since $R$ is symmetric, we infer the existence of an orthogonal matrix $V\in\R^{\stateDim\times\stateDim}$ such that
	\begin{gather*}
		V^\top R V =\vcentcolon \begin{bmatrix}
			R_1 & 0\\
			0 & 0
		\end{bmatrix}, \qquad
		V^\top \delayHam V =\vcentcolon \begin{bmatrix}
			\delayHam_1 & \delayHam_2\\
			\delayHam_2^\top & \delayHam_4
		\end{bmatrix}, \\
		V^\top \tilde{Z} V =\vcentcolon \begin{bmatrix}
			Z_1 & Z_2\\
			Z_3 & Z_4
		\end{bmatrix}
	\end{gather*}
	with $R_1 \in \spd{r}$, resulting in
	\begin{align*}
		0 &\leq \begin{bmatrix}
			R_1 - \delayHam_1 & - \delayHam_2 & Z_1 & Z_2\\
			- \delayHam_2^\top & -\delayHam_4 & Z_3 & Z_4\\
			Z_1^\top & Z_3^\top & \delayHam_1 & \delayHam_2\\
			Z_2^\top & Z_4^\top & \delayHam_2^\top & \delayHam_4 
		\end{bmatrix}.
	\end{align*}
	We conclude $\delayHam_4=0$ and $\delayHam_2=0$, implying $\Kern(R) \subseteq \Kern(\delayHam)$ and $\Kern(R) \cap \image(S) = \{0\}$, and consequently $Z_2 = 0$, $Z_3 = 0$, and $Z_4=0$, which implies $\Kern(R) \subseteq \Kern(Z)$ and~\eqref{eqn:pHcondition:necessaryConditions:b}. The inequality \eqref{eqn:pHcondition} is thus equivalent to
\begin{equation}\label{eqn:trafoschurR}
	\begin{bmatrix}
		R_1 - \delayHam_1 & Z_1 \\
		Z_1^\top & \delayHam_1 
	\end{bmatrix} \geq 0.
\end{equation}
Since $\rank(\delayHam) = \rank(\delayHam_1)$, it exists orthogonal $U \in \R^{s \times s} $ such that
\begin{gather*}
	U^\top R_1 U = \begin{bmatrix}
		R_{11} & 0\\
		0 & R_{14}
	\end{bmatrix}, \quad
	U^\top \delayHam_1 U = \begin{bmatrix}
		\delayHam_{11} & 0\\
		0 & 0
	\end{bmatrix}\\
	U^\top Z_1 U = \begin{bmatrix}
		Z_{11} & Z_{12}\\
		Z_{13} & Z_{14} 
	\end{bmatrix}
\end{gather*}
with diagonal matrices $R_{11}, \delayHam_{11} \in \GL{s}{\R}$ and $R_{14} \in \GL{s-r}{\R}$ invertible and diagonal. 
Applied to \eqref{eqn:trafoschurR} we get
\begin{equation*}
	\begin{bmatrix}
		R_{11}-\delayHam_{11} & 0 & Z_{11} & Z_{12}\\
		0 & R_{14} & Z_{13} & Z_{14}\\
		Z_{11}^\top & Z_{13}^\top & \delayHam_{11} & 0\\
		Z_{12}^\top & Z_{14}^\top & 0 & 0 
	\end{bmatrix} \geq 0
\end{equation*}
We conclude $Z_{12}=0$ and $Z_{14}=0$, which in turn implies $\Kern(\delayHam) \subseteq \Kern(Z)$. 
\end{proof}

Using the geometric analysis, we can continue to establish existence results using the condensed form~\eqref{eqn:trafoschurR} with $R_1\in\spd{r}$. Based on the necessary conditions from \Cref{lem:pHdelayCondition:RZrelation}, we can assume $\delayHam_1\in\spd{r}$. Since $R_1$ is symmetric positive definite, Silvester's law of inertia implies the existence of $V_1\in\GL{r}{\R}$ such that $V_1^\top R_1 V_1 = I_r$ and set $\widehat{Z}_1 \vcentcolon= V_1^\top Z_1 V$. To simplify our computations, we make the choice 
\begin{equation}
	\label{eqn:parametrizationS}
	\delayHam_1\colon \R\to\R^{r\times r},\qquad \alpha \mapsto \alpha V_1^{-\top}V_1^{-1}.
\end{equation}
We immediately obtain $R_1-\delayHam_1(\alpha) = (1-\alpha)V_1^{-\top}V_1^{-1}\in\spd{r}$ and $\delayHam_1(\alpha)\in\spd{r}$ for $\alpha\in(0,1)$. Assuming $\alpha\in(0,1)$, the Schur complement $R_1 - \delayHam_1(\alpha) - Z_1^\top \delayHam_1^{-1}(\alpha)Z_1$ is symmetric positive semi-definite if and only if
\begin{equation}
	\label{eqn:SchurComplementConditionS}
	 (\alpha-\alpha^2) I_r - \widehat{Z}^\top \widehat{Z}\in\spsd{r}.
\end{equation}
Let $\sigma_{\max}$ denote the largest singular value of $\widehat{Z}$. Then, it is easily seen that the Schur complement condition~\eqref{eqn:SchurComplementConditionS} is satisfied if and only if $\sigma_{\max} \leq \tfrac{1}{4}$. In this case, the choice $\delayHam_1(\tfrac{1}{2})$ guarantees that the condition~\eqref{eqn:pHcondition} is satisfied. We summarize our discussion in the following theorem.

\begin{theorem}
	\label{thm:LyapKravoskiiConstruction}
	Consider a delay equation of the form~\eqref{eqn:pHDelay} and assume $\ker(R)\subseteq \ker(Z)$ and $\ker(R)\cap \image(Z) = \{0\}$. Let $r \vcentcolon= \rank(R)$ and $V_1\in\R^{\stateDim\times r}$ such that $V_1^\top R V_1 = I_r$. If $\|V_1^\top Z V_1\|_2 \leq 1$, then for $\delayHam\vcentcolon= \tfrac{1}{2}R\in\spsd{\stateDim}$ the condition~\eqref{eqn:pHcondition} is satisfied. 
\end{theorem}

\begin{proof}
	Let $V_2\in\R^{\stateDim\times(\stateDim-r)}$ be a basis of $\ker(R)$ and define $V \vcentcolon= \begin{bmatrix}
		V_1 & V_2
	\end{bmatrix}\in\GL{\stateDim}{\R}$.
	Set $\delayHam \vcentcolon= \tfrac{1}{2}R \in\spsd{\stateDim}$ and define $\widehat{Z} \vcentcolon= V_1^\top Z V_1$. We then obtain (similarly as in the proof of \Cref{lem:pHdelayCondition:RZrelation})
	\begin{align}
		\label{eqn:pHcondition:transformed}
		\begin{bmatrix}
			V^\top & 0\\
			0 & V^\top
		\end{bmatrix}\begin{bmatrix}
			R-\delayHam & \tfrac{1}{2}Z\\
			\tfrac{1}{2}Z^\top & \delayHam
		\end{bmatrix}\begin{bmatrix}
			V & 0\\
			0 & V
		\end{bmatrix}
		&= \tfrac{1}{2}\begin{bmatrix}
			I_r & 0 & \widehat{Z} & 0\\
			0 & 0 & 0 & 0\\
			\widehat{Z} & 0 & I_r & 0\\
			0 & 0 & 0 & 0
		\end{bmatrix}.
	\end{align}
	Consider an orthogonal $U\in\GL{r}{\R}$ that diagonalizes $\widehat{Z}^\top \widehat{Z}$, i.e., $U^\top \widehat{Z}^\top \widehat{Z} U = \diag(\eta_1,\ldots,\eta_r)$ with $\eta_i \geq \eta_{i+1} \geq 0$ for $i=1,\ldots,r-1$. By assumption, we have $\eta_1 \leq 1$. The result now follows from
	\begin{displaymath}
		I_r - \widehat{Z}^\top \widehat{Z} = U(I_r - \diag(\eta_1,\ldots,\eta_r))U^\top \in\spsd{r},
	\end{displaymath}
	which is two times the Schur complement of the non-zero submatrix on the right-hand side in~\eqref{eqn:pHcondition:transformed}.
\end{proof}

Unfortunately, the condition $\|V_1^\top Z V_1\|_2 \leq 1$ is only sufficient, but not necessary, as the following example suggests.

\begin{example}
	Let $n=2$ and consider the matrices
	\begin{displaymath}
		R = \begin{bmatrix}
			1 & 0\\
			0 & 1
		\end{bmatrix}\qquad\text{and}\qquad Z = \begin{bmatrix}
			0 & \tfrac{1}{\sqrt{3}}\\
			\tfrac{2}{\sqrt{3}} & 0
		\end{bmatrix}.
	\end{displaymath}
	With $V_1 = I_2$ we obtain $\|V_1^\top ZV_1\|_2 = \|Z\|_2 = \tfrac{2}{\sqrt{3}} > 1$. Nevertheless, it is easy to see that with the choice $\delayHam = \diag(\tfrac{1}{2},\tfrac{1}{4})$
	the condition~\eqref{eqn:pHcondition} is satisfied.
\end{example}

\subsection{Comparison with the literature}
Using the results from the previous subsection, we are now in a position to relate our definition of a delay \pH system with the passivity matrix inequality from \cite[Lem.~1]{NicL01}, presented in \Cref{lem:passiveTimeDelay}. We obtain the following result.

\begin{proposition}
	\label{prop:delayPHinequality}
	Consider a \pH delay system~\eqref{eqn:pHDelay} and assume $S\in\spd{\stateDim}$. Then, \eqref{eqn:pHDelay} fulfills the inequality~\eqref{eqn:NicL01inequality}. 
\end{proposition}

\begin{proof}
	A \pH delay system of the form \eqref{eqn:pHDelay} corresponds to the delay equation~\eqref{eqn:delaySys} with
	\begin{align*}
		A_0 &= H^{-1}(J-R), &
		A_1 &= - H^{-1}Z, &
		G &= H^{-1}B. 
	\end{align*}
	With the choice $Q = \tfrac{1}{2}H\in\spd{\stateDim}$ we obtain
	\begin{multline*}
		A_0^\top Q + Q A_0 + Q A_1 \delayHam^{-1} A_1^\top Q +S\\
		= -R + \delayHam + \tfrac{1}{4} Z \delayHam^{-1}Z^\top,
	\end{multline*}
	which is the negative of the Schur complement of~\eqref{eqn:pHcondition}, and hence negative semi-definite.
\end{proof}

Two remarks are in order:
\begin{itemize}
	\item On the one hand, the \pH condition~\eqref{eqn:pHcondition} is slightly more general, since it does not require $\delayHam$ to be nonsingular. Moreover, the explicit choice of $H$ (and thus $Q$), makes the matrix inequality~\eqref{eqn:NicL01passivityinequ} easier to verify; see the discussion in \cref{subsec:timeDelayPH:implications}.
	\item On the other hand, the explicit choice of $H$ has an impact on $\ker(R)$ and thus on the possible choices of~$\delayHam$; cf.~\Cref{lem:pHdelayCondition:RZrelation}. In particular, the dimension of $\ker(R)$ is not independent of~$H$, but $\dim(\ker(Z))$ and $\dim(\image(Z))$ are independent of~$H$, which renders~\eqref{eqn:pHcondition} less flexible then~\eqref{eqn:NicL01passivityinequ}.
\end{itemize}

\subsection{Delayed feedback control}
\label{subsec:delayedFeedback}

As important special case, we consider the standard \pH system~\eqref{eqn:pH} and assume a delayed feedback of the form
\begin{equation}
	\label{eqn:delayedOutputFeedback}
	\inpVar(t) \vcentcolon= -F\outVar(t-\tau) + v(t)
\end{equation}
with $F\in\R^{\stateDim\times\stateDim}$, then results in the time-delay system
\begin{equation}
	\label{eqn:pH:timeDelayedFeedback}
	\begin{aligned}
		H\dot{\state}(t) &= (J-R)\state(t) - G F G^\top\state(t-\tau) + v(t),\\
		\outVar(t) &= G^\top \state(t),
	\end{aligned}
\end{equation}
which is (formally) of the form~\eqref{eqn:pHDelay} with $Z = G F G^\top$. We now discuss necessary and sufficient conditions on $G$ and $F$ such that~\eqref{eqn:pH:timeDelayedFeedback} is indeed a \pH delay system. We notice that the necessary condition from \Cref{lem:pHdelayCondition:RZrelation} is satisfied if
\begin{equation}
  \begin{aligned}
    \label{eqn:RZrelationsWithFeedback}
	\ker(R) \subseteq \ker(G^\top) &= \{0\},\\
	\ker(R) \cap \image(G) &= \{0\}.
  \end{aligned}
\end{equation}
Let us mention, that the second condition also shows up in the characterization of optimal controls for \pH systems, see \cite[Thm.~8]{SchPFWM21}. Assuming that \eqref{eqn:RZrelationsWithFeedback} holds and using \Cref{thm:LyapKravoskiiConstruction} we thus obtain an upper bound on the norm of the feedback gain~$F$ such that we can guarantee that the closed-loop system is still a delay \pH system.

\section{Conclusions}

We have presented a novel definition of a time-delay \pH system, which is based on a detailed investigation of the corresponding infinite-dimensional \KYP inequality. We then showed that our class of delay \pH systems satisfies important properties of \pH systems, such as delay-independent passivity and invariance under interconnection. Moreover, the structured \pH form allows for a simplified construction of Lyapunov-Krasovskii functional compared to standard results in the literature. Based on this first step, we envision multiple generalizations, including neutral and nonlinear delay systems and an extension to delay differential-algebraic equations.


\bibliographystyle{plain-doi}
\bibliography{literature}

\end{document}